\documentclass[12pt, reqno]{amsart}

\usepackage{amssymb,amsrefs}
\usepackage{hyperref}
\usepackage{setspace}

\numberwithin{equation}{section}
\theoremstyle{plain}
\newtheorem{thm}{Theorem} [section]
\newtheorem{prp}{Proposition} [section]
\newtheorem{cor}{Corollary} [section]
\newtheorem{lem}{Lemma} [section]
\newtheorem{defin}{Definition} [section]
\newtheorem{exa}{Example} [section]

\newcommand{\e}{\varepsilon}
\newcommand{\N}{\mathbb{N}}
\newcommand{\C}{\mathbb{C}}

\begin{document}

\title{$(s,p)$-Valent Functions}

\begin{abstract}
We introduce the notion of $(\mathcal F,p)$-valent functions. We concentrate in our investigation on the case, where $\mathcal F$ is the class of polynomials of degree at most $s$. These functions, which we call $(s,p)$-valent functions, provide a natural generalization of $p$-valent functions (see~\cite{Ha}). We provide a rather accurate characterizing of $(s,p)$-valent functions in terms of their Taylor coefficients, through ``Taylor domination'', and through linear non-stationary recurrences with uniformly bounded coefficients. We prove a ``distortion theorem" for such functions, comparing them with polynomials sharing their zeroes, and obtain an essentially sharp Remez-type inequality in the spirit of~\cite{Y3} for complex polynomials of one variable. Finally, based on these results, we present a Remez-type inequality for $(s,p)$-valent functions.
\end{abstract}

\author{Omer Friedland}
\address{Institut de Math\'ematiques de Jussieu, Universit\'e Pierre et Marie Curie (Paris 6), 4 Place Jussieu, 75005 Paris, France.}
\email{omer.friedland@imj-prg.fr}
\thanks{}

\author{Yosef Yomdin}
\address{Department of Mathematics, The Weizmann Institute of Science, Rehovot 76100, Israel.}
\email{yosef.yomdin@weizmann.ac.il}

\maketitle

\section{Introduction} \label{Intro}

Let us introduce the notion of ``$(\mathcal F,p)$-valent functions". Let $\mathcal F$ be a class of functions to be precise later. A function $f$ regular in a domain $\Omega\subset\C$ is called $(\mathcal F,p)$-valent in $\Omega$ if for any $g\in\mathcal F$ the number of solutions of the equation $f(z) = g(z)$ in $\Omega$ does not exceed $p$.

\smallskip

For example, the classic $p$-valent functions are obtained for $\mathcal F$ being the class of constants, these are functions $f$ for which the equation $f = c$ has at most $p$ solutions in $\Omega$ for any $c$. There are many other natural classes $\mathcal F$ of interest, like rational functions, exponential polynomial, quasi-polynomials, etc. In particular, for the class ${\mathcal R}_s$ consisting of rational functions $R(z)$ of a fixed degree $s$, the number of zeroes of $f(z)-R(z)$ can be explicitly bounded for $f$ solving linear ODEs with polynomial coefficients (see, e.g.\ \cite{Bin}). Presumably, the collection of $({\mathcal R}_s,p)$-valent functions with explicit bounds on $p$ (as a function of $s$) is much wider, including, in particular, ``monogenic'' functions (or ``Wolff-Denjoy series'') of the form $f(z) = \sum_{j = 1}^\infty {{\gamma_j}\over {z-z_j}}$ (see, e.g.\ \cites{Mar,Sib} and references therein).

\smallskip

However, in this note we shall concentrate on another class of functions, for which $\mathcal F$ is the class of polynomials of degree at most $s$. We denote it in short as $(s,p)$-valent functions. For an $(s,p)$-valent function $f$ the equation $f = P$ has at most $p$ solutions in $\Omega$ for any polynomial $P$ of degree $s$. We shall always assume that $p\ge s + 1$, as subtracting from $f$ its Taylor polynomial of degree $s$ we get zero of order at least $s + 1$. Note that this is indeed a generalization of $p$-valent functions, simply take $s = 0$, and every $(0,p)$-valent function is $p$-valent.

\smallskip

As we shall see this class of $(s,p)$-valent functions is indeed rich and appears naturally in many examples: algebraic functions, solutions of algebraic differential equations, monogenic functions, etc. In fact, it is fairly wide (see Section~\ref{Taylor.Dom}). It possesses many important properties: Distortion theorem, Bernstein-Markov-Remez type inequalities, etc. Moreover, this notion is applicable to any analytic function, under an appropriate choice of the domain $\Omega$ and the parameters $s$ and $p$. In addition, it may provide a useful information in very general situations.

\smallskip

The following example shows that an $(s,p)$-valent function may be not $(s + 1,p)$-valent:

\begin{exa} \label{exa}
Let $f(x) = x^p + x^N$ for $N\ge 10p + 1$. Then, for $s = 0,\dots,p-1$, the function $f$ is $(s,p)$-valent in the disk $D_{1/ 3}$, but only $(p,N)$-valent there.

Indeed, taking $P(x) = x^p + c$ we see that the equation $f(x) = P(x)$ takes the form $x^N = c$. So for $c$ small enough, it has exactly $N$ solutions in the $D_{1/ 3}$. Now, for $s = 0,\dots,p-1$, take a polynomial $P(x)$ of degree $s\le p-1$. Then, the equation $f(x) = P(x)$ takes the form $x^p -P(x) + x^N = 0$. Applying Lemma~3.3 of~\cite{Y1} to the polynomial $Q(x) = x^p -P(x)$ of degree $p$ (with leading coefficient $1$) we find a circle $S_\rho = \{|x| = \rho\}$ with ${1/ 3}\le\rho\le {1/ 2}$ such that $| Q(x) |\ge (1/2)^{10p}$ on $S_\rho$. On the other hand $x^N\le (1/2)^{10p + 1} < (1/2)^{10p}$ on $S_\rho$. Therefore, by the Rouch\'e principle the number of zeroes of $Q(x) + x^N$ in the disk $D_\rho$ is the same as for $Q(x)$, which is at most $p$. Thus, $f$ is $(s,p)$-valent in the disk $D_{1/ 3}$, for $s = 0,\dots,p-1$.
\end{exa}

This paper is organized as follows: in Section~\ref{Taylor.Dom} we characterize $(s,p)$-valent functions in terms of their Taylor domination and linear recurrences for their coefficients. In Section~\ref{Distortion theorem} we prove a Distortion theorem for $(s,p)$-valent functions. In Section~\ref{CP} we make a detour and investigate Remez-type inequalities for complex polynomials, which is interesting in its own right. Finally, in Section~\ref{Remez.Ineq}, we extend the Remez-type inequality to $(s,p)$-valent functions, via the Distortion theorem.

\section{Taylor domination, bounded recurrences} \label{Taylor.Dom}

In this section we provide a rather accurate characterization of $(s,p)$-valent functions in a disk $D_R$ in terms of their Taylor coefficients. ``Taylor domination'' for an analytic function $f(z) = \sum_{k = 0}^\infty a_k z^k$ is an explicit bound of all its Taylor coefficients $a_k$ through the first few of them. This property was classically studied, in particular, in relation with the Bieberbach conjecture: for univalent $f$ we always have $|a_k| \le k|a_1|$ (see~\cites{Bie,Bi,Ha} and references therein). To give an accurate definition, let us assume that the radius of convergence of the Taylor series for $f$ is $\hat R$, for $0<\hat{R}\le + \infty$.

\begin{defin} [Taylor domination] \label{def:T.D} \label{def:domination}
Let $0<R <\hat{R}$, $N\in\N$, and $S(k)$ be a positive sequence of a subexponential growth. The function $f$ is said to possess an $(N,R,S(k))$-Taylor domination property if
$$
|a_k| R^k\le S(k)\max_{i = 0,\dots,N}| a_i |R^i~,\quad k\ge N+1.
$$
\end{defin}

The following theorem shows that $f$ is an $(s,p)$-valent function in $D_R$, essentially, if and only if its lower $s$-truncated Taylor series possesses a $(p-s,R,S(k))$-Taylor domination.

\begin{thm} \label{thm:SP.TD}
Let $f(z) = \sum_{k = 0}^\infty a_k z^k$ be an $(s,p)$-valent function in $D_R$, and let $\hat f(z) = \sum_{k = 1}^\infty a_{s + k} z^k$ be the lower $s$-truncation of $f$. Put $m = p-s$. Then, $\hat f$ possesses an $(m,R,S(k))$-Taylor domination, with $S(k) = \left(\frac{A_mk}{m}\right)^{2m}$, and $A_m$ being a constant depending only on $m$.

\smallskip

Conversely, if $\hat f$ possesses an $(m,R,S(k))$-Taylor domination, for a certain sequence $S(k)$ of a subexponential growth, then for $R' < R$ the function $f$ is $(s,p)$-valent in $D_{R'}$, where $p = p(s + m,S(k),R'/R)$ depends only on $m + s$, the sequence $S(k)$, and the ratio $R'/R$. Moreover, $p$ tends to $\infty$ for $R'/R\to 1$, and it is equal to $m + s$ for $R'/R$ sufficiently small.
\end{thm}

\begin{proof}
First observe that if $f$ is $(s,p)$-valent in $D_R$, then $\hat f$ is $m$-valent there, with $m = p-s$. Indeed, put $P(z) = \sum_{k = 0}^{s}a_k z^k + cz^s$, with any $c\in {\mathbb C}$. Then, $f(z)-P(z) = z^s(\hat f(z)-c)$ may have at most $p$ zeroes. Consequently, $\hat f(z)-c$ may have at most $m$ zeroes in $D_R$, and thus $\hat f$ is $m$-valent there. Now we apply the following classic theorem:

\begin{thm} [Biernacki, 1936,~\cite{Bi}] \label{thm:bier}
If $f$ is $m$-valent in the disk $D_{R}$ of radius $R$ centered at $0\in{\mathbb{C}}$ then
$$
|a_k| R^k\le\left(\frac{A_m k}{m}\right)^{2m}\max_{i = 1,\dots,m}|a_{i}|R^{i} ~, \quad k\ge m+1,
$$
where $A_m$ is a constant depending only on $m$.
\end{thm}

In our situation, Theorem~\ref{thm:bier} claims that the function $\hat f$ which is $m$-valent in $D_{R}$, possesses an $(m,R,\left(\frac{A_m k}{m}\right)^{2m})$-Taylor domination property. This completes the proof in one direction.

\smallskip

In the opposite direction, for polynomial $P(z)$ of degree $s$ the function $f-P$ has the same Taylor coefficients as $\hat f$, starting with the index $k = s + 1$. Consequently, if $\hat f$ possesses an $(m,R,S(k))$-Taylor domination, then $f-P$ possesses an $(s + m,R,S(k))$-Taylor domination. Now a straightforward application of Theorem~2.3 of~\cite{BaYo} provides the required bound on the number of zeroes of $f-P$ in the disk $D_R$.
\end{proof}

A typical situation for natural classes of $(s,p)$-valent functions is that they are $(s,p)$-valent for any $s$ with a certain $p = p(s)$ which depends on $s$. However, it is important to notice that essentially $\it any$ analytic function possesses this property, with some $p(s)$.

\begin{prp} \label{Any.SP}
Let $f(z)$ be an analytic function in an open neighbourhood $U$ of the closed disk $D_R$. Assume that $f$ is not a polynomial. Then, the function $f$ is $(s,p(s))$-valent for any $s$ with a certain sequence $p(s)$.
\end{prp}

\begin{proof}
Let $f$ be given by its Taylor series $f(z) = \sum_{k = 0}^\infty a_k z^k$. By assumptions, the radius of convergence $\hat R$ of this series satisfies $\hat R > R$. Since $f$ is not a polynomial, for any given $s$ there is the index $k(s)>s$ such that $a_{k(s)}\ne 0$. We apply now Proposition~1.1 of~\cite{BaYo} to the lower truncated series $\hat f(z) = \sum_{k = 1}^\infty a_{s + k}z^k$. Thus, we obtain, an $(m,\hat R, S(k))$-Taylor domination for $\hat f$, for certain $m$ and $S(k)$. Now, the second part of Theorem~\ref{thm:SP.TD} provides the required $(s,p(s))$-valency for $f$ in the smaller disk $D_R$, with $p(s) = p(s + m,S(k),{R}/{\hat R})$.
\end{proof}

More accurate estimates of $p(s)$ can be provided via the lacunary structure of the Taylor coefficients of $f$. Consequently, $(s,p)$-valency becomes really interesting only for those {\it classes} of analytic functions $f$ where we can specify the parameters in an explicit and uniform way. The following theorem provides still very general, but important such class.

\begin{thm} \label{thm:rec}
Let $f(z) = \sum_{k = 0}^\infty a_k z^k$ be $(s,s + m)$-valent in $D_R$ for any $s$. Then, the Taylor coefficients $a_k$ of $f$ satisfy a linear homogeneous non-stationary recurrence relation
\begin{align} \label{eq:Main.Rec}
a_k = \sum_{j = 1}^m c_j(k) a_{k-j}
\end{align}
with uniformly bounded (in $k$) coefficients $c_j(k)$ satisfying $|c_j(k)| \le C\rho^j$, with $C = e^2A_m^{2m},\rho = R^{-1}$, where $A_m$ is the constant in the Biernacki's Theorem~\ref{thm:bier}.

\smallskip

Conversely, if the Taylor coefficients $a_k$ of $f$ satisfy recurrence relation~\eqref{eq:Main.Rec}, with the coefficients $c_j(k)$, bounded for certain $K,\rho >0$ and for any $k$ as $|c_j(k)| \le K\rho^j$, $j = 1,\dots,m$, then for any $s$, $f$ is $(s,s + m)$-valent in a disk $D_R$, with $R = {1\over {2^{3m + 1}(2K + 2)\rho}}$.
\end{thm}

\begin{proof}
Let us fix $s\ge 0$. As in the proof of Theorem~\ref{thm:SP.TD}, we notice that if $f$ is $(s,s + m)$-valent in $D_R$, then its lower $s$-truncated series $\hat f$ is $m$-valent there. By Biernacki's Theorem~\ref{thm:bier} we conclude that
$$
|a_{s + m + 1}|R^{m + 1} \le \left( {{A_m(m + 1)}\over m} \right)^{2m}\max_{i = 1,\dots,m}|a_{s + i}|R^{i}\le C\max_{i = 1,\dots,m}|a_{s + i}|R^{i},
$$
with $C = e^2A_m^{2m}$. Putting $k = s + m + 1$, and $\rho = R^{-1}$ we can rewrite this as
$$
|a_k| \le C\max_{j = 1,\dots,m}|a_{k-j}|\rho^{j}.
$$

Hence we can chose the coefficients $c_j(k)$, $k = s + m + 1$, in such a way that $a_k = \sum_{j = 1}^m c_j(k) a_{k-j}$, and $|c_j(k)| \le C\rho^j$. Notice that the bound on the recursion coefficients is sharp, and take $f(z) = [1-({z\over R})^m]^{-1}$ (in this case, as well as for other lacunary series with the gap $m$, the coefficients $c_j(k)$ are defined uniquely). This completes one direction of the proof.

\smallskip

In the opposite direction, the result follows directly from Theorem~4.1 of~\cite{BaYo}, and Lemma~2.2.3 of~\cite{Roy.Yom}, with $R = {1\over {2^{3m + 1}(2K + 2)\rho}}$.
\end{proof}

\section{Distortion theorem} \label{Distortion theorem}

In this section we prove a distortion-type theorem for $(s,p)$-valent functions which shows that the behavior of these functions is controlled by the behavior of a polynomial with the same zeroes.

\smallskip

First, let us recall the following theorem for $p$-valent functions, which is our main tool in proof.

\begin{thm} \cite[Theorem~5.1]{Ha} \label{thm:h}
Let $g(z) = a_0 + a_1z + \dots$ be a regular non-vanishing $p$-valent function in $D_1$. Then, for any $z\in D_1$
$$
\left( \frac{1-|z|}{1 + |z|} \right)^{2p}\le |g(z)/a_0|\le \left(\frac{1 + |z|}{1-|z|} \right)^{2p}.
$$
\end{thm}

Now, we are at the point to formulate a distortion-type theorem for $(s,p)$-valent functions.

\begin{thm} [Distortion theorem] \label{thm:disto}
Let $f$ be an $(s,p)$-valent function in $D_1$ having there exactly $s$ zeroes $x_1,\dots, x_s$ (always assumed to be counted according
to multiplicity). Define a polynomial
$$
P(x) = A\prod_{j = 1}^s(x-x_j),
$$
where the coefficient $A$ is chosen such that the constant term in the Taylor series for $f(x)/P(x)$ is equal to $1$. Then, for any $x\in D_1$
$$
\left( \frac{1-|x|}{1 + |x|} \right)^{2p}\le |f(x)/P(x)|\le \left( \frac{1 + |x|}{1-|x|} \right)^{2p}.
$$
\end{thm}

\begin{proof}
The function $g(x) = f(x)/P(x)$ is regular in $D_1$ and does not vanish there. Moreover, $g$ is $p$-valent in $D_1$. Indeed, the equation $g(x) = c$ is equivalent to $f(x) = cP(x)$ so it has at most $p$ solutions by the definition of $(s,p)$-valent functions. Now, apply Theorem~\ref{thm:h} to the function $g$.
\end{proof}

It is not clear whether the requirement for $f$ to be $(s,p)$-valent is really necessary in this theorem. The ratio $g(x) = {\frac{f(x)}{P(x)}}$ certainly may be not $p$-valent for $f$ being just $p$-valent, but not $(s,p)$-valent. Indeed, take $f(x) = x^p + x^N$ as in Example~\ref{exa}. By this example $f$ is $p$-valent in $D_{1/ 3}$ and it has a root of multiplicity $p$ at zero. So $g(x) = {{f(x)}/ {x^p}} = 1 + x^{N-p}$ and the equation $g(x) = c$ has $N-p$ solutions in $D_{1/ 3}$ for $c$ sufficiently close to $1$. So $g$ is not $p$-valent there.

\section{Complex polynomials} \label{CP}

The distortion theorem~\ref{thm:disto}, proved in the previous section, allows us easily to extend deep properties from polynomials to $(s,p)$-valent functions, just by comparing them with polynomials having the same zeros. In this section we make a detour and investigate one specific problem for complex polynomials, which is interesting in its own right: a Remez-type inequality for complex polynomial (compare~\cite{Re,Y3}). Denote by
$$
V_{\rho}(g) = \{z : |g(z)|\le\rho\}
$$
the $\rho$ sub-level set of a function $g$. For polynomials in one complex variable a result similar to the Remez inequality is provided by the classic Cartan (or Cartan-Boutroux) lemma (see, for example, \cite{Gor} and references therein):

\begin{lem} [Cartan's lemma~\cite{Ca}, in form of~\cite{Gor}]
Let $\alpha,\e>0$, and let $P(z)$ be a monic polynomial of degree $d$. Then
$$
V_{\e^d}(P)\subset\cup_{j = 1}^p D_{r_j},
$$
where $p\le d$, and $D_{r_1},\dots,D_{r_p}$ are balls with radii $r_j>0$ satisfying $\sum_{j = 1}^p r_j^\alpha\le e (2\e)^\alpha$.
\end{lem}

In~\cites{Bru, Bru.Bru2, Zer2, Zer1} some generalizations of the Cartan-Boutroux lemma to plurisubharmonic functions have been obtained, which lead, in particular, to the bounds on the size of sub-level sets. In these lines in~\cite{Bru} some bounds for the covering number of sublevel sets of complex analytic functions have been obtained, similar to the results of~\cite{Y3} in the real case. Now, we shall derive from the Cartan lemma both the definition of the invariant $c_{d,\alpha}$ and the corresponding Remez inequality.

\begin{defin}
Let $Z\subset D_1$. The $(d,\alpha)$-Cartan measure of $Z$ is defined as
$$
c_{d,\alpha}(Z) = \min \left\{ \left( \sum_{j = 1}^p r_j^\alpha \right)^{1/\alpha} : \text{there is a cover of $Z$ by $p\le d$ balls with radii $r_j>0$} \right\}.
$$
\end{defin}

Note that the $\alpha$-dimensional Hausdorff content of $Z$ is defined in a similar way
$$
H_\alpha(Z) = \inf \left\{ \sum_j r_j^\alpha : \text{there is a cover of $Z$ by balls with radii $r_j>0$} \right\}.
$$

Thus, by the above definitions, we have $H^{1\over\alpha}_\alpha(Z)\le c_{d,\alpha}(Z)$.

\smallskip

For $\alpha = 1$ the $(d,1)$-Cartan measure $c_{d,\alpha}(Z)$ was introduced and used, under the name ``$d$-th diameter'', in~\cite{CP1,CP}. In
particular, Lemma~3.3 of~\cite{CP1} is, essentially, equivalent to the case $\alpha = 1$ of our Theorem~\ref{thm:C}. In Section
~\ref{geom.Cart} below we provide some initial geometric properties of $c_{d,\alpha}(Z)$ and show that a proper choice of $\alpha$ may improve
geometric sensitivity of this invariant.

\smallskip

Now we can state and proof our generalized Remez inequality for complex polynomials:

\begin{thm} \label{thm:C}
Let $P(z)$ be a polynomial of degree $d$. Let $Z\subset D_1$. Then, for any $\alpha > 0$
$$
\max_{D_1}|P(z)|\le \left( \frac{6e^{1/\alpha}}{c_{d,\alpha}(Z)} \right)^d\max_{Z}|P(z)| \le \left( \frac{6e}{H_\alpha(Z)} \right)^{d\over\alpha}\max_{Z}|P(z)|.
$$
\end{thm}

\begin{proof}
Assume that $|P(z)|\le 1$ on $Z$. First, we prove that the absolute value $A$ of the leading coefficient of $P$ satisfies 
$$
A\le \left( \frac{2e^{1/\alpha}}{c_{d,\alpha}(Z)} \right)^d.
$$

Indeed, we have $Z\subset V_{1}(P)$. By the definition of $c_{d,\alpha}(Z)$ for every covering of $V_{1}(P)$ by $p$ disks $D_{r_1},\dots,D_{r_p}$ of the radii $r_1,\dots,r_d$ (which is also a covering of $Z$) we have $\sum_{i = 1}^d r_i^\alpha\ge c_{d,\alpha}(Z)^\alpha$. Denoting, as above, the absolute value of the leading coefficient of $P(z)$ by $A$ we have by the Cartan lemma that for a certain covering as above
$$
c_{d,\alpha}(Z)^\alpha\le\sum_{i = 1}^d r_i^\alpha\le e \left( \frac{2}{A^{1/ d}} \right)^\alpha.
$$

Now, we write $P(z) = A\prod_{j = 1}^d (z-z_j)$, and consider separately two cases:

\smallskip

1) All $|z_j|\le 2$. Thus, $\max_{D_1} |P(z)|\le A 3^d\le \left( \frac{2e^{1/\alpha}}{c_{d,\alpha}(Z)} \right)^d 3^d$.

\smallskip

2) For $j = 1,\dots,d_1 < d$, $|z_j|\le 2$, while $|z_j| > 2$ for $j = d_1 + 1,\dots,d$. Denote
$$
P_1(z) = A\prod_{j = 1}^{d_1} (z-z_j) ~,\quad P_2(z) = \prod_{j = d_1 + 1}^d (z-z_j),
$$
and notice that for any two points $v_1,v_2\in D_1$ we have $|{{P_2(v_1)}/ {P_2(v_2)}}| < 3^{d-d_1}$. Consequently we get
$$
\frac{\max_{D_1} |P(z)|}{\max_{Z}|P(z)|} < 3^{d-d_1}\frac{\max_{D_1}|P_1(x)|}{\max_{Z}|P_1(z)|}.
$$
All the roots of $P_1$ are bounded in absolute value by $2$, so by first part we have
$$
\frac {\max_{D_1}|P_1(z)|}{\max_{Z}|P_1(z)|}\le \left( \frac{2e^{1/\alpha}}{c_{d,\alpha}(Z)} \right)^d 3^{d_1}.
$$
Application of the inequality $H_\alpha(Z)\le c_{d,\alpha}(Z)^\alpha$ completes the proof.
\end{proof}

Let us stress a possibility to chose an optimal $\alpha$ in the bound of Theorem~\ref{thm:C}. Let
$$
K_d(Z) = \inf_{\alpha>0} \left( \frac{6e^{1/\alpha}}{c_{d,\alpha}(Z)} \right)^d ~,\quad K_d^H(Z) = \inf_{\alpha>0} \left( \frac{6e}{H_\alpha(Z)} \right)^{d\over\alpha}.
$$

\begin{cor} \label{cor:thmC}
Let $P(z)$ be a polynomial of degree $d$. Let $Z\subset D_1$. Then,
$$
\max_{D_1}|P(z)|\le K_d(Z)\max_{Z}|P(z)|\le K_d^H(Z)\max_{Z}|P(z)|.
$$
\end{cor}

\subsection{Geometric and analytic properties of the invariant $c_{d,\alpha}$} \label{geom.Cart}

Clearly, the invariant $c_{d,\alpha}(Z)$ is monotone in $Z$, that is, for $Z_1\subset Z_2$ we have $c_{d,\alpha}(Z_1) \le c_{d,\alpha}(Z_2)$. Also, for any $Z$, we have

\begin{prp} \label{d.points}
Let $\alpha >0$. Then, $c_{d,\alpha}(Z) > 0$ if and only if $Z$ contains more than $d$ points. In the latter case, $c_{d,\alpha}(Z)$ is greater than or equal to one half of the minimal distance between the points of $Z$.
\end{prp}

\begin{proof}
Any $d$ points can be covered by d disks with arbitrarily small radii. But, the radius of at least one disk among $d$ disks covering more than $d + 1$ different points is greater than or equal to the one half of a minimal distance between these points.
\end{proof}

The lower bound of Proposition~\ref{d.points} does not depend on $\alpha$. However, in general, this dependence is quite prominent.

\begin{exa}
Let $Z = [a,b]$. Then, for $\alpha\ge 1$ we have $c_{d,\alpha}(Z) = (b-a)/2$, while for $\alpha \le 1$ we have $c_{d,\alpha}(Z) = d^{{1\over\alpha}-1}(b-a)/2$.

Indeed, in the first case the minimum is achieved for $r_1 = (b-a)/2,r_2 = \dots = r_d = 0$, while in the second case for $r_1 = r_2 = \dots = r_d = (b-a)/2d$.
\end{exa}

\begin{prp} \label{compar.alpha}
Let $\alpha >\beta >0$. Then, for any $Z$
\begin{align} \label{eq:compar.alpha}
c_{d,\alpha}(Z) \le c_{d,\beta}(Z) \le d^{({1\over\beta}-{1\over\alpha})}c_{d,\alpha}(Z).
\end{align}
\end{prp}

\begin{proof}
Let $r = (r_1,\dots,r_d)$ and $\gamma >0$. Consider $||r||_\gamma = (\sum_{j = 1}^d r_j^\gamma)^{1\over\gamma}$. Then, by the definition, $c_{d,\gamma}(Z)$ is the minimum of $||r||_\gamma$ over all $r = (r_1,\dots,r_d)$ being the radii of $d$ balls covering $Z$. Now we use the standard comparison of the norms $||r||_\gamma$, that is, for any $x = (x_1,\dots,x_d)$ and for $\alpha >\beta >0$,
$$
||x||_\alpha \le ||x||_\beta \le d^{({1\over\beta}-{1\over\alpha})}||x||_\alpha.
$$

Take $r = (r_1,\dots,r_d)$ for which the minimum of $||r||_\beta$ is achieved, and we get
$$
c_{d,\alpha}(Z) \le ||r||_\alpha \le ||r||_\beta = c_{d,\beta}(Z).
$$

Now taking $r$ for which the minimum of $||r||_\alpha$ is achieved, exactly in the same way we get the second inequality.
\end{proof}

Now, we compare $c_{d,\alpha}(Z)$ with some other metric invariants which may be sometimes easier to compute. In each case we do it for the most convenient value of $\alpha$. Then, using the comparison inequalities of Proposition~\ref{compar.alpha}, we get corresponding bounds on $c_{d,\alpha}(Z)$ for any $\alpha > 0$. In particular, we can easily produce a simple lower bound for $c_{d,2}(Z)$ through the measure of $Z$:

\begin{prp} \label{measure}
For any measurable $Z\subset D_1$ we have
$$
c_{d,2}(Z)\ge ({{\mu_2(Z)}/ {\pi}})^{1/2}.
$$
\end{prp}

\begin{proof}
For any covering of $Z$ by $d$ disks $D_1,\dots,D_d$ of the radii $r_1,\dots,r_d$ we have $\pi(\sum_{i = 0}^d r^2_i)\ge\mu_2(Z)$.
\end{proof}

However, in order to deal with discrete or finite subsets $Z\subset D_1$ we have to compare $c_{d,\alpha}(Z)$ with the covering number $M(\e,Z)$ (which is, by definition, the minimal number of $\e$-disks covering $Z$).

\begin{defin}
Let $Z\subset D_1$. Define
$$
\omega_{cd}(Z) = \sup_\e\e(M(\e,Z)-d)^{1/ 2} ~,\quad \rho_d(Z) = d\e_0,
$$
where $\e_0$ is the minimal $\e$ for which there is a covering of $Z$ with $d$ $\e$-disks. Note that, writing $y = M(\e,Z) = \Psi(\e)$, and taking the inverse $\e = \Psi^{-1}(y)$, we have $\e_0 = \Psi^{-1}(d)$.
\end{defin}

As it was mentioned above, a very similar invariant
$$
\omega_d(Z) = \sup_\e\e(M(\e,Z)-d)
$$
was introduced and used in~\cite{Y3} in the real case. We compare $\omega_{cd}$ and $\omega_d$ below.

\begin{prp}
Let $Z\subset D_1$. Then, $\omega_{cd}(Z)/2\le c_{d,2}(Z)\le c_{d,1}(Z)\le\rho_d(Z)$.
\end{prp}

\begin{proof}
To prove the upper bound for $c_{d,1}(Z)$ we notice that it is the infimum of the sum of the radii in all the coverings of $Z$ with $d$ disks, while $\rho_d(Z)$ is such a sum for one specific covering.

\smallskip

To prove the lower bound, let us fix a covering of $Z$ by $d$ disks $D_i$ of the radii $r_i$ with $c_{d,2}(Z) = (\sum_{i = 0}^d r_i^2)^{1/2}$. Let $\e > 0$. Now, for any disk $D_j$ with $r_j\ge\e$ we need at most $4r^2_j /\e^2$ $\e$-disks to cover it. For any disk $D_j$ with $r_j\le\e$ we need exactly one $\e$-disk to cover it, and the number of such $D_j$ does not exceed $d$. So, we conclude that $M(\e,Z)$ is at most $d + (4/\e^2)\sum_{i = 0}^d r_i^2$. Thus, we get $c_{d,2}(Z) = (\sum_{i = 0}^d r_i^2)^{1/2}\ge {\e / 2} (M(\e,Z)-d)^{1/2}$. Taking supremum with respect to $\e>0$ we get $c_{d,2}(Z)\ge\omega_{cd}(Z)/2$.
\end{proof}

Since $M(\e,Z)$ is always an integer, we have
$$
\omega_d(Z)\ge\omega_{cd}(Z).
$$

For $Z\subset D_1$ of positive plane measure, $\omega_d(Z) = \infty$ while $\omega_{cd}(Z)$ remains bounded (in particular, by $\rho_d(Z)$).

\smallskip

Some examples of computing (or bounding) $\omega_d(Z)$ for ``fractal" sets $Z$ can be found in~\cite{Y3}. Computations for $\omega_{cd}(Z)$ are essentially the same. In particular, in an example given in~\cite{Y3} in connection to~\cite{Fav} we have that for $Z = Z_r = \{1,{1 / {2^r}},{1 / {3^r}},\dots,{1 / {k^r}},\dots\}$
$$
\omega_d(Z_r)\asymp\frac{r^r}{(r + 1)^{r + 1} d^r} ~,\quad\omega_{cd}(Z_r)\asymp\frac{(2r + 1)^r}{(2r + 2)^{r + 1} d^{r + 1/2}}.
$$

The asymptotic behavior here is for $d\to\infty$, as in~\cite{Fav}.

\subsection{An example}

We conclude this section with one very specific example. Let
$$
Z = Z(d,h) = \{z_1,z_2,\dots,z_{2d-1},z_{2d}\}~, \quad x_i\in {\mathbb C}, ~d\ge 2.
$$

We assume that
$Z$ consists of $d$, $2\eta$-separated couples of points, with points in each couple being in a distance $2h$. Let $2D(Z)$ be the diameter of the smallest disk containing $Z$, where $h\ll 1$, and $2\eta\gg h$.

\begin{prp} \label{example}
Let $Z$ be as above. Then,
\begin{enumerate}
\item $\omega_d(Z) = dh$.
\item $\omega_{cd}(Z) = \sqrt dh$.
\item For $\alpha >0$, we have $c_{d,\alpha}(Z) \le d^{1\over\alpha}h$.
\item For $\alpha \gg 1$, we have $c_{d,\alpha}(Z) = d^{1\over\alpha}h$.
\item For $\kappa = [\log_d({{D(Z)}\over h})]^{-1}$, we have $c_{d,\kappa}(Z)\ge\eta$.
\end{enumerate}
\end{prp}

\begin{proof}
For $\e > h$, we have $M(\e,Z)\le d$, and hence $M(\e,Z)- d$ is negative. For $\e < h$, we have $M(\e,Z) = 2d$, and $M(\e,Z)- d = d$. Thus the supremum of $\e(M(\e,Z)- d)$, or the supremum of $\e(M(\e,Z)- d)^{1\over 2}$, is achieved as $\e < h$ tends to $h$. Therefore, $\omega_d(Z) = dh$, and $\omega_{cd}(Z) = \sqrt dh$.

\smallskip

Covering each couple with a separate ball of radius $h$, we get for any $\alpha >0$ that $c_{d,\alpha}(Z)\le d^{1\over\alpha}h$. For $\alpha \gg 1$ it is easy to see that this uniform covering is minimal. Thus, for such $\alpha$ we have the equality $c_{d,\alpha}(Z) = d^{1\over\alpha}h$.

\smallskip

Now let us consider the case of a ``small'' $\alpha = \kappa$. Take a covering of $Z$ with certain disks $D_j$, $j \le d$. If there is at least one disk $D_j$ containing three points of $Z$ or more, the radius of this disk is at least $\eta$. Thus, for this covering $(\sum_{j = 1}^d r_j^\kappa)^{1\over\kappa}\ge\eta$. If each disk in the covering contains at most two points, it must contain exactly two, otherwise these disks could not cover all the $2d$ points of $Z$. Hence, the radius of each disk $D_j$ in such covering is at least $h$, an their number is exactly $d$. We have, by the choice of $\kappa$, that $(\sum_{j = 1}^d r_j^{\kappa})^{1\over {\kappa}}\ge d^{1\over {\kappa}}h = D(Z)\ge\eta$.
\end{proof}

Proposition~\ref{example} shows that $c_{d,1}(Z)\le dh$, while we have $c_{d,\kappa}(Z)\ge\eta$. So using $\alpha = 1$ and $\alpha = \kappa$ in the Remez-type inequality of Theorem~\ref{thm:C} we get two bounds for the constant $K_d(Z):$
\begin{align} \label{boundsK}
K_d(Z) \le \left( \frac{6e}{dh} \right)^d \quad \text{or} \quad K_d(Z) \le \left( \frac{6e^{1/\kappa}}{\eta} \right)^d.
\end{align}

But $e^{1/\kappa} = e^{\log_d({{D(Z)}\over h})} = ({{D(Z)}\over h})^{1\over {\ln d}}$. So the second bound of \eqref{boundsK} takes a form
$$
K_d(Z) \le \left( \frac{6D(Z)}{\eta^{\ln d} h} \right)^{d\over {\ln d}}.
$$

We see that for $d\ge 3$ and for $h\to 0$ the asymptotic behavior of this last bound, corresponding to $\alpha = \kappa$, is much better than of the first bound in \eqref{boundsK}, corresponding to $\alpha = 1$. Notice, that $\kappa$ depends on $h$ and $D(Z)$, i.e.\ on the specific geometry of the set $Z$.

\section{Remez inequality} \label{Remez.Ineq}

Now, we present a Remez-type inequality for $(s,p)$-valent functions. We recall that by Proposition~\ref{Any.SP} above, any analytic function in an open neighborhood $U$ of the closed disk $D_R$ is $(s,p(s))$-valent in $D_R$ for any $s$ with a certain sequence $p(s)$. Consequently, the following theorem provides a non-trivial information for any analytic function in an open neighborhood of the unit disk $D_1$. Of course, this results becomes really interesting only in cases where we can estimate $p(s)$ explicitly.

\begin{thm}
Let $f$ be an analytic function in an open neighborhood $U$ of the closed disk $D_1$. Assume that $f$ has in $D_1$ exactly $s$ zeroes, and that it is $(s,p)$-valent in $D_1$. Let $Z$ be a subset in the interior of $D_1$, and put $\rho = \rho(Z) = \min\{\eta : Z\subset D_\eta\}$. Then, for any $R < 1$ function $f$ satisfies
$$
\max_{D_{R}} |f(z)|\le\sigma_p(R,\rho) K_s(Z)\max_{Z}|f(z)|,
$$
where $\sigma_p(R,\rho) = \left( \frac{1 + R}{1-R}\cdot\frac{1 + \rho}{1-\rho} \right)^{2p}$.
\end{thm}

\begin{proof}
Assume that $|f(x)|$ is bounded by $1$ on $Z$. Let $x_1,\dots, x_s$ be zeroes of $f$ in $D_1$. Consider, as in Theorem~\ref{thm:disto}, the polynomial
$$
P(x) = A\prod_{j = 1}^l(x-x_j),
$$
where the coefficient $A$ is chosen in such a way that the constant term in the Taylor series for $g(x) = f(x)/P(x)$ is equal to $1$. Then by Theorem~\ref{thm:disto} for $g$ we have
$$
\left( \frac{1-|x|}{1 + |x|} \right)^{2p}\le | g(x) |\le \left( \frac{1 + |x|}{1-|x|} \right)^{2p}.
$$

We conclude that $P(x)\le ({\frac{1 + \rho}{1-\rho}})^{2p}$ on $Z$. Hence by the polynomial Remez inequality provided by Theorem~\ref{thm:C} we obtain
$$
|P(x)|\le K_s(Z) \left( \frac{1 + \rho}{1-\rho} \right)^{2p}
$$
on $D_1$. Finally, we apply once more the bound of Theorem~\ref{thm:disto} to conclude that
$$
|f(x)| \le K_s(Z) \left( \frac{1 + R}{1-R} \right)^{2p} \left(\frac{1 + \rho}{1-\rho} \right)^{2p}
$$
on $D_{R}$.
\end{proof}

\end{document}